\documentclass[10pt,a4paper]{article}

\usepackage[final]{optional}

\usepackage{amsfonts, amsmath, wasysym}
\usepackage{amssymb,amsthm,
paralist
}

\usepackage{
latexsym,
}

\usepackage[usenames]{color}

\usepackage{url}

\definecolor{darkgreen}{rgb}{0,0.5,0}
\definecolor{darkred}{rgb}{0.7,0,0}
\usepackage[colorlinks, 
citecolor=darkgreen, linkcolor=darkred
]{hyperref}

\theoremstyle{plain}
\newtheorem{lemma}{Lemma}[section]
\newtheorem{thm}[lemma]{Theorem}
\newtheorem{prop}[lemma]{Proposition}
\newtheorem{cor}[lemma]{Corollary}

\theoremstyle{definition}
\newtheorem{defn}[lemma]{Definition}

\newtheorem{conj}[lemma]{Conjecture}

\newtheorem{rmk}[lemma]{Remark}

\numberwithin{equation}{section}

\newcommand{\m}{\ensuremath{{\cal M}}}
\newcommand{\n}{\ensuremath{{\cal N}}}

\newcommand{\cc}{\ensuremath{{\cal C}}}

\newcommand{\cl}{\ensuremath{{\cal L}}}

\newcommand{\pl}[2]{{\frac{\partial #1}{\partial #2}}}

\newcommand{\plndd}[3]{{\frac{\partial^2 #1}
{{\partial #2}{\partial #3}}}}

\newcommand{\al}{\alpha}
\newcommand{\be}{\beta}
\newcommand{\ga}{\gamma}
\newcommand{\Ga}{\Gamma}
\newcommand{\de}{\delta}
\newcommand{\om}{\omega}
\newcommand{\Om}{\Omega}

\newcommand{\la}{\lambda}

\newcommand{\Si}{\Sigma}

\newcommand{\Th}{\Theta}

\newcommand{\vph}{\varphi}
\newcommand{\ep}{\varepsilon}
\newcommand{\Ups}{\Upsilon}

\newcommand{\R}{\ensuremath{{\mathbb R}}}
\newcommand{\N}{\ensuremath{{\mathbb N}}}

\newcommand{\C}{\ensuremath{{\mathbb C}}}

\newcommand{\downto}{\downarrow}

\newcommand{\Wedge}{{\Lambda}}
\newcommand{\tensor}{\otimes}
\newcommand{\lap}{\Delta}

\newcommand{\grad}{\nabla}

\newcommand{\beq}{\begin{equation}}
\newcommand{\eeq}{\end{equation}}
\newcommand{\beqa}{\begin{equation}\begin{aligned}}
\newcommand{\eeqa}{\end{aligned}\end{equation}}
\newcommand{\brmk}{\begin{rmk}}
\newcommand{\ermk}{\end{rmk}}
\newcommand{\partref}[1]{\hbox{(\csname @roman\endcsname{\ref{#1}})}}
\newcommand{\half}{\frac{1}{2}}

\title{{\sc 
the canonical expanding soliton and harnack inequalities 
for ricci flow
}
\\ 
}
\author{Esther Cabezas-Rivas and Peter M. Topping}
\date{\today~(@\the\time mpm)}

\begin{document}

\maketitle
\parskip=10pt

\newcommand{\di}{\partial_i}
\newcommand{\djj}{\partial_j}
\newcommand{\dk}{\partial_k}
\newcommand{\dl}{\partial_l}
\newcommand{\tr}{{\rm tr}}
\newcommand{\lie}{{\cal L}}
\newcommand{\Rop}{{\cal R}}
\newcommand{\Rm}{{\mathrm{Rm}}}
\newcommand{\Ric}{{\mathrm{Ric}}}
\newcommand{\Rc}{{\mathrm{Rc}}}
\newcommand{\RicO}{\overset{\circ}{\Ric}}
\newcommand{\RS}{{\mathrm{R}}}
\newcommand{\Hess}{{\mathrm{Hess}}}
\newcommand{\hess}{{\mathrm{hess}}}
\newcommand{\f}{\ensuremath{{\cal F}}}
\newcommand{\ghat}{\hat{g}}
\newcommand{\fhat}{\hat{f}}
\newcommand{\bop}{+}
\newcommand{\avint}{{\int\!\!\!\!\!\!-}}
\newcommand{\Hbb}{\ensuremath{{\mathbb H}}}

\newcommand{\lexp}{{\cl_{\tau_1,\tau_2}\exp_x}}
\newcommand{\ta}{\ensuremath{{\tau_1}}}
\newcommand{\tb}{\ensuremath{{\tau_2}}}
\newcommand{\lc}{\ensuremath{{\cl Cut}}}
\newcommand{\lcslice}{\ensuremath{{\lc_{\ta,\tb}}}}
\newcommand{\cg}{\check g}
\def\({\left (}
\def \){\right)}

\begin{abstract}
We introduce the notion of Canonical Expanding Ricci Soliton, and
use it to derive new Harnack inequalities
for Ricci flow. This viewpoint also gives geometric insight into
the existing Harnack inequalities of Hamilton and Brendle.
\end{abstract}

\section{Introduction}
\label{intro}

Recently, in \cite{CT1}, we introduced the notion of
Canonical Soliton. 
Roughly speaking, given any Ricci flow on a manifold 
\m\ over a time interval
$I\subset (-\infty,0)$, we imagined the time direction as
an additional space direction and constructed a 
shrinking Ricci soliton on $\m\times I$
with respect to a completely new time direction.

Considering these solitons in the context of known 
notions and theorems in Riemannian geometry then 
induced interesting concepts and results concerning 
the original Ricci flow, many of
which were first discovered by Perelman \cite{P1}.
For example, considering geodesic distance in our soliton
metrics gives rise to Perelman's $\cl$-length.

It is also fruitful to consider existing Ricci flow
theory applied to the Canonical Soliton flows.
For example, applying theory of 
McCann and the second author and Ilmanen \cite{MT} is one way of 
leading to the results of \cite{Lopt} which ultimately recovers
essentially all of the monotonic quantities for Ricci flow
used by Perelman \cite{P1}.
See \cite{grenoble} for a broader description.

In this paper we describe a slight variation of the
Canonical Shrinking Solitons -- namely the Canonical Expanding
Solitons -- which have completely different applications.
These new solitons are adapted to explaining and proving
Harnack inequalities in the spirit of the original result
of Hamilton \cite{hamharnack} and the more recent result of
Brendle \cite{brendleharnack}. Our work recovers both of these
known Harnack inequalities, and gives new ones too.
(See Theorem \ref{harnackthm}.)
In addition, our method explains clearly what is behind
a Harnack inequality: it is simply the assertion that a
given curvature condition is preserved on the Canonical
Expanding Soliton.

As a by-product of our work we give an answer to the
question of Wallach and Hamilton \cite{formations}
which asks for a geometric construction whose curvature
is represented by the matrix Harnack quantity of 
Hamilton \cite{hamharnack}. This question prompted the
pioneering works of Chow-Chu \cite{CC1} (see also the relevant modification in \cite[chapter 11, \S 1.3]{CLN}) and Chow-Knopf \cite{ChK}
which led in turn to the constructions of
\cite[\S 6]{P1} and \cite{CT1}. Each of these papers
constructed an object whose curvatures were similar to the
Hamilton Harnack quantities, sometimes modulo either missing terms
or the changed signs one finds in Perelman's versions
of Hamilton's Harnack quantities \cite[\S 6, \S 7]{P1}.

It turns out that the Riemannian curvatures of the Canonical Expanding
Solitons give rise to the exact Hamilton matrix Harnack
quantities, with the closest parallels in the existing literature
to be found within the approximation approach in \cite[\S 4]{CC1} and, more recently, the calculations of Brendle \cite{brendleharnack}. 
In our case, for example, the full Hamilton Harnack inequality
is equivalent to the Canonical Expanding Soliton having
weakly positive curvature operator (in an appropriate limit).
Our new Harnack inequalities can be phrased in terms
of this soliton satisfying other natural curvature
conditions.

The paper is organised as follows.
In Section 2 we introduce the Canonical Expanding Solitons
and the flow they induce, describe their asymptotics and
explain how they lead to Harnack inequalities. Our new
Harnack inequalities are stated in Section \ref{resultssect}.
In Section 3 we give the rigorous proof of the Harnack inequalities
stated in Theorem \ref{harnackthm}. 
To do that we use the Canonical Expanding Solitons to derive
the equations satisfied by their (limiting) curvature 
$\Rop_\infty$, and then piece together an argument which
borrows much from the work of Hamilton \cite{hamharnack}
and Brendle \cite{brendleharnack}.
The computations for the curvature of the Canonical Expanding
Solitons are given in Appendix \ref{computations}, and 
in Appendix \ref{wilking_appendix} we describe a recent
insight of Burkhard Wilking \cite{wilking} which allows one
to construct some new invariant curvature cones.

For an introduction to Ricci flow, we refer to 
\cite{RFnotes}; for further background on Canonical Solitons
and an overview of the use of Harnack inequalities in Ricci flow,
see \cite{grenoble}.

\emph{Acknowledgements:} We thank Burkhard Wilking, 
Simon Brendle and Mario Micallef for useful conversations. 
Both authors are
supported by The Leverhulme Trust. The first author was also partially supported by the DGI (Spain) and FEDER Project MTM2007-65852, and by the net REAG MTM2008-01013-E.

\section{The Canonical Expanding Solitons}

\subsection{Definitions and basic properties}

\begin{thm}
\label{CESthm}
Suppose $g(t)$ is a  Ricci flow, i.e. a
solution of 
\begin{equation}
\label{RFeq}
\pl{g}{t}=-2\,\Ric(g(t))
\end{equation}
defined on a manifold $\m$ of dimension $n\in\N$, for 
$t$ within a time interval $[0,T]$, with uniformly bounded
curvature.
Suppose $N\in\N$ is sufficiently large
to give a positive definite metric $\check g_N$ (which
we normally write simply as $\check g$) on
$\check\m:=\m\times (0,T]$ defined by
$$\check g_{ij}=\frac{g_{ij}}{t}; 
\qquad 
\check g_{00}=\frac{N}{2 t^3}+\frac{R}{t}+\frac{n}{2 t^2};
\qquad
\check g_{0i}=0,$$
where $i,j$ are coordinate indices on the \m\ factor, 
$0$ represents the index of the time coordinate $t\in (0,T]$, 
and the scalar curvature of $g$ is written as $R$.

Then up to errors of order $\frac{1}{N}$, the metric $\check g$
is a gradient expanding Ricci soliton on the higher dimensional
space $\check\m$:
\beq
\label{CES1}
E_N:=
\Ric(\check g)+\Hess_{\check g} \left(-\frac{N}{2t}\right)
+ \half \check g \simeq 0,
\eeq
by which we mean that for any $k\in \{0,1,2,\ldots\}$
the quantity 
\begin{equation}
\label{meaning}
N \left[\check\grad^k
E_N
\right]
\end{equation}
is bounded uniformly locally on $\check \m$\ (independently of $N$) 
where $\check\grad$ is the Levi-Civita
connection corresponding to $\check g$.
\end{thm}

This construction should be compared to the Canonical Shrinking
Solitons of \cite{CT1}. 
Various signs have changed, and every $\tau$ has been replaced with
a $t$. We have also strengthened the sense in which \eqref{CES1}
is to hold in order to make it more useful in rigorous proofs,
and considered Ricci flows defined all the way down to $t=0$
as is relevant in the study of Harnack estimates.

The proof reduces to computing all relevant quantities
explicitly. We give the results of these computations,
including exact expressions for the full curvature tensor
of $\cg$ in Appendix \ref{computations}.

We now introduce a new time parameter $s\in (0,1]$ and consider
the flow of metrics 
$$G(s):=s\psi_s^*(\check g)$$
where $\psi_s:\check\m\to\check\m$ 
is the family of maps, diffeomorphic onto their images, 
generated by integrating the collection of vector fields
\beq
\label{Xformula}
\begin{aligned}
X_s&:=\frac{1}{s}\check\grad\(-\frac{N}{2t}\)\\
&=\frac{t}{s}\pl{}{t}
-\check g^{00}\(R+\frac{n}{2t}\)\frac{1}{s}\pl{}{t}
\end{aligned}
\eeq
starting with 
$\psi_1=identity$. 
If we imagine $N$ to be large, then $X_s$ is approximately
the vector field $\frac{t}{s}\pl{}{t}$, which could be
integrated on the whole of $\m\times (0,\infty)$ 
to give the map 
$\psi^\infty_s:\m\times (0,\infty)\to\m\times (0,\infty)$
defined by
$$\psi^\infty_s(x,t)=(x,st).$$
A slightly closer inspection yields:

\begin{prop}
\label{psiprop}
The map $\psi_s$ converges smoothly in its arguments
$x\in\m$, $t\in (0,T]$ and $s\in (0,1]$
to (an appropriate restriction of) $\psi^\infty_s$
as $N\to\infty$.
\end{prop}

If $\cg$ were an exact Ricci soliton metric, then $G(s)$
would be an exact Ricci flow on $\m\times (0,T]$
for $s\in (0,1]$
(see \cite[\S 1.2.2]{RFnotes}). A minor adjustment of
the standard theory reveals:
\begin{prop}
\label{Geqprop}
The flow $G(s)$ satisfies
$$\pl{G}{s}=-2\Ric(G(s))+\psi_s^*(E_N),$$
where $E_N$ is defined in Theorem \ref{CESthm}.
\end{prop}

Since $G(s)$ is essentially a Ricci flow, we can imagine
that curvature conditions such as positive curvature
operator which are preserved for Ricci flows might
be essentially preserved for $G(s)$.
If we can argue that $G(s)$ satisfies such a condition
in the limit $s\downto 0$, then we should deduce the
same condition for $G(1)=\cg$, and the idea is that 
a curvature condition for $\cg$ is the geometric way
of expressing a Harnack inequality.
We will see precise assertions along these lines later,
but for now, this reasoning justifies the following section.

\subsection{Asymptotics of the Canonical Expanding Soliton}
\label{CESasymp}

We will argue that in the limit $t\downto 0$, the 
Canonical Solitons are conical.
Near $t=0$, the dominant term in the definition of $\check g_{00}$
is emphatically $\frac{N}{2t^3}$. If we neglect the other
terms for the moment, and change variables from $t$ to
$r:=t^{-\half}$, then for large $r$ (small $t$) $\check g$ 
can be written approximately as
\beq
\label{checkgasymp}
\begin{aligned}
\check g_N &\sim 
\frac{g(t)}{t}
+ \frac{N}{2t^3}dt^2\\
&= r^2g(r^{-2}) +2Ndr^2
\end{aligned}
\eeq
and this suggests that asymptotically the Canonical Soliton opens 
like the cone
$$\Si_N:=(\m\times (0,\infty),\frac{g(0)}{t} + \frac{N}{2t^3}dt^2)
= (\m\times (0,\infty), r^2 g(0) + 2N dr^2)$$
(using coordinates $(x,t)$ or $(x,r)$ respectively
on $\m\times(0,\infty)$)
with shallow cone angle for large $N$. This motivates:

\begin{lemma}
If we define $\bar G(s):=s(\psi_s^\infty)^*(\check g_N)$
on $\m\times (0,T/s]$, then
$$\bar G(s)\to \frac{g(0)}{t}+ \frac{N}{2t^3}dt^2$$
smoothly locally on $\m\times (0,\infty)$ as $s\downto 0$.
\end{lemma}

\begin{proof}
It may be clearest to make a third change of coordinates,
writing $\al=\ln t$. Then 
$$\frac{g(0)}{t}+ \frac{N}{2t^3}dt^2=
e^{-\al}\(g(0)+\frac{N}{2}d\al^2\)$$
and the map $\psi^\infty_s$, viewed in $(x,\al)$ coordinates
as a map $\m\times\R\to\m\times\R$, corresponds to 
$$\psi^\infty_s(x,\al)=(x,\al+\ln s).$$
Moreover,
$$\cg_N=e^{-\al}\(g(e^\al)+\frac{N}{2}d\al^2\)
+\(Re^\al+\frac{n}{2}\)d\al^2,$$
and so 
$$\bar G(s)=e^{-\al}\(g(e^{\al+\ln s})+\frac{N}{2}d\al^2\)
+\(R s^2 e^{\al}+s\frac{n}{2}\)d\al^2,$$
\end{proof}

One consequence of this is 
that for any $x\in\m$ and sequence $t_i\downto 0$, 
we have convergence of pointed rescalings of the Canonical
Expanding Soliton:
$$(\check \m,t_i\check g_N,(x,t_i))\to (\Si_N,(x,\al=0))$$
as $i\to\infty$ 
in the sense of Cheeger-Gromov convergence \cite[\S 7]{RFnotes}.
It is also relevant to note that
$$(\Si_N,(x,\al=0))\to (\Si_\infty,(x,0))$$
as $N\to\infty$, where 
$$\Si_\infty := (\m,g(0))\times\R,$$
i.e. the cone straightens out to a cylinder.
In practice, we will be most interested in taking limits
of geometric quantities as $N\to\infty$ (despite the 
fact that the metric itself will degnerate) 
and then as $s\downto 0$
rather than the other way round.

\subsection{Using Canonical Expanding Solitons to give
Harnack inequalities}

We are now in a position to elaborate on the use of
Theorem \ref{CESthm} to obtain Harnack estimates.
It is well known that certain curvature conditions
are preserved under Ricci flow. For example, a
Ricci flow on a closed manifold 
which starts with weakly positive curvature
operator will also satisfy this property at later times.

Given a Ricci flow for which we would like a Harnack inequality,
the trick, effectively, is to apply this preservation principle 
not to the Ricci flow itself, but to its Canonical Expanding
Soliton. Given the asymptotics discussed in the previous section,
we should study Ricci flows $g(t)$ on an $n$-manifold $\m$, for which the curvature of $(\m,g(0))\times\R$ lies in a subspace of the
space of all possible curvatures which is
preserved under $(n+1)$-dimensional Ricci flow.
If we were to make the leap of faith that such a curvature condition
should also be preserved under the approximate Ricci flow
generated by the (approximate, incomplete) Canonical Soliton, then
we would deduce that $G(1)=\cg$ should satisfy this condition,
and this statement can be considered to be a Harnack inequality.

In order to make a precise statement of the Harnack inequalities,
we will pass to the limit $N\to\infty$. The Canonical Soliton
metrics $\cg_N$ degenerate in this limit, although if we view them as
metrics on $T^*\check\m$ rather than on $T\check\m$,
then they converge to a weakly positive definite 
tensor $\check g_\infty\in Sym^2(T\check\m)$
whose only nonzero components are $(\check g_\infty)^{ij}=tg^{ij}$.
More importantly, some of the geometric quantities such as curvature
associated with $\cg_N$ behave well in the
same limit, and the actual Harnack estimates will be
statements about them. 

Let V be a (real) vector space of dimension $m$. 
We call $\Rop\in \tensor^4 V^*$ an \emph{algebraic curvature tensor}
if it satisfies the symmetries of the curvature tensor of 
a Riemannian metric, including the first Bianchi identity (cf. Chapter V of \cite{KN}).
Given a manifold $\n^m$, we use the same terminology to describe a section 
of $\tensor^4 T^*\n$ which is an algebraic curvature tensor
in each fibre.

\begin{prop}
\label{Rminfinityprop}
In the setting of Theorem \ref{CESthm}, the full curvature
tensor $\Rop(\cg_N)$ of $\cg_N$, viewed as a section of
$\tensor^4 T^*\check\m$,
converges smoothly as $N\to\infty$ to a limit 
algebraic curvature tensor $\Rop_\infty$.
At a point $(x,t)$ in $\check \m$, the coefficients of the
tensor $\Rop_\infty$ are given in terms of the coefficients
of the curvature of $g(t)$ at the point $x$ by
\begin{align*}
{\stackrel{\infty}{R}}_{ijkl} &  = \frac1{t} R_{ijkl} \\
{\stackrel{\infty}{R}}_{i0j0} & =  \frac1{t}\(\lap R_{ij} + 2 R_{ikjl} R^{kl}   -R_i^k R_{jk} + \frac{R_{ij}}{2 t} - \frac1{2} \nabla_i \nabla_j R\)\\
{\stackrel{\infty}{R}}_{ij0k} & = \frac1{t}\(\nabla_i R_{jk} - \nabla_j R_{ik}\) 
\end{align*} 
Moreover, the Ricci curvature $\Ric(\cg_N)\in\Ga(Sym^2T^*\check\m)$
converges to a limit $\Ric_\infty$ determined by
$$\Ric_\infty\left(X+\pl{}{t},X+\pl{}{t}\right)=
\Ric_{g(t)}(X,X)+\langle X,\grad R\rangle_{g(t)}
+\half\left(\pl{R}{t}+\frac{R}{t}\right),$$
for each $X\in T\m$.
\end{prop}

The proposition follows immediately from Appendix \ref{computations}.

An algebraic curvature tensor with the same coefficients as $\Rop_\infty$ 
arises as the limit (as $\varepsilon \to \infty$ and $\delta \to 0$) of the Riemannian curvature associated to the two-parameter family of Riemannian metrics $\tilde g_{\varepsilon, \delta}$ introduced in \cite[\S4]{CC1}. Compare also with the definition of the $(0, 4)$-tensor $S$ in \cite{brendleharnack}.

The curvature converges also when viewed as any tensor of
type $(p,q)$ with $p+q=4$, but because the metric 
$\cg_N$ is degenerating in the limit, the assertion above
carries the most information.

The symmetries of an algebraic curvature tensor on a vector
space $V$, say,
allow one to see it as a symmetric bilinear form on $\Wedge^2 V$,
constrained further by the Bianchi identity.
We sometimes emphasise this viewpoint by calling 
such a constrained element of $Sym^2(\Wedge^2V^*)$
an \emph{(algebraic) curvature (bilinear) form},
and write the entire space of such forms $Sym_B^2(\Wedge^2V^*)$.
We will occasionally use the same notation for the space of 
algebraic curvature tensors, that is, 
we will switch between these two viewpoints 
implicitly,  often without changing notation.

Viewing $\Rop_\infty$ as a section of $Sym_B^2(\Wedge^2 T^*\check\m)$
(using the obvious extension of the notation above)
the full matrix Harnack inequality of Hamilton \cite{hamharnack}
is precisely equivalent to 
$\Rop_\infty$
being (weakly) positive definite. (This would normally be referred
to as weakly positive curvature operator, but as mentioned above,
because the limit
metric $\check g_\infty$ is degenerate, that would be a weaker assertion.)  
Hamilton's trace Harnack inequality is precisely equivalent to
$\Ric_\infty$ being weakly positive definite.
Brendle's Harnack inequality \cite{brendleharnack} is precisely equivalent to
$\Rop_\infty$ lying in a certain cone introduced by
Brendle and Schoen \cite{BS} 
(see Appendix \ref{wilking_appendix}).

Finally, if one were to assume that all Harnack inequalities
arose in this way, one would conclude that the long-sought Harnack 
inequality for Ricci flows $g(t)$ on 3-manifolds with weakly positive
Ricci curvature would be unreasonable.
One would deduce that $(\m,g(0))\times\R$ has weakly positive
Ricci curvature, but the framework of this paper would then require 
weakly positive Ricci curvature to be preserved for 4-dimensional
Ricci flows, which is false.

In the next section we find some other preserved curvature
conditions which yield new Harnack inequalities, and in
Section \ref{actsect} we explain to what extent one might
expect to generalise further.

\subsection{New Harnack inequalities}
\label{resultssect}

As we have indicated, our Harnack inequalities will be phrased 
in terms of the limiting curvature form $\Rop_\infty$
lying in appropriate subsets of $Sym^2_B(\Wedge^2T^*\check\m)$.
We now work towards defining some examples of such subsets.

Following \cite{micallef_moore} and \cite{wilking} 
we consider the complexified tangent bundle $T^\C\n$
of a manifold $\n$, and implicitly 
extend each curvature form complex linearly
to act on complexified 2-vectors -- i.e. elements of $\Wedge^2T^\C\n$.
To such a 2-vector $\om$, one can associate a rank defined
to be the least number $k\in\{0,1,2,\dots\}$ 
such that in each fibre
we can write $\om$ as a complex linear combination of at most
$k$ simple elements $u\wedge v$ with $u,v\in T^\C\n$.

With this viewpoint, we can define the following convex
cones within $Sym^2_B(\Wedge^2T^*\n)$, which arise from
`Wilking' cones as discussed in Appendix \ref{wilking_appendix}:
$$\cc_k(\n):=
\{\Rop\in Sym^2_B(\Wedge^2T^*\n)\ |\ \Rop(\om,\bar\om)\geq 0
\text{ for all }\om\in\Wedge^2T^\C\n
\text{ of rank no more than }k\}.$$

We will see that these cones 
are invariant under Ricci flow in the sense that if
$g(t)$ is a Ricci flow on a closed manifold $\m$ for $t\in [0,T]$
and $\Rop(g(0))\in \cc_k(\m)$, then $\Rop(g(t))\in \cc_k(\m)$
for all $t\in [0,T]$.
Note that the following theorem does not require $\m$
to be closed.

\begin{thm} (Main Harnack Theorem.)
\label{harnackthm}
Suppose $g(t)$ is a complete Ricci flow on a manifold $\m$ for 
$t\in (0,T]$ with scalar curvature uniformly bounded from
above, and for some $k\in\N$, 
$\Rop(g(t))\in \cc_k(\m)$ for all $t\in (0,T]$. 
Then $\Rop_\infty\in \cc_k(\check\m)$.
\end{thm}

\brmk
\label{traceresult}
Note that the cones $\cc_k$ are nested in the sense that
if $k_1\leq k_2$ then $\cc_{k_1}(\n)\supset \cc_{k_2}(\n)$.
Moreover, if $\Rop_\infty\in \cc_k(\check\m)$ for some
$k$ (even $k=1$) then $\Ric_\infty$ is positive 
definite, so we deduce Hamilton's trace Harnack inequality.
By definition, a metric whose curvature tensor lies 
in $\cc_1(\n)$ has weakly positive sectional curvature,
so the hypothesised upper bound for the scalar curvature
implies a uniform upper and lower bound on the full curvature
tensor.
\ermk

The highest possible rank of an element $\om\in\Wedge^2T\n$
is $[\frac{dim(\n)}{2}]$, and therefore if $k=[\frac{n+1}{2}]$,
the cone $\cc_k(\check\m)$ is precisely the weakly positive
definite elements of $Sym^2_B(\Wedge^2T^*\check\m)$, and we
recover the main result of \cite{hamharnack} stated in
a geometric form:

\begin{cor} (Equivalent to Hamilton \cite[Main Theorem]{hamharnack}.)
Suppose $g(t)$ is a complete Ricci flow on a manifold $\m$ for 
$t\in (0,T]$ with uniformly bounded curvature.
If $\Rop(g(t))$ is weakly positive definite for all $t\in (0,T]$
then $\Rop_\infty$ is weakly positive definite.
\end{cor}

The fact that the cone of weakly positive definite curvature
forms is preserved under Ricci flow is also, separately,
due to Hamilton \cite{ham4mfd}. 

The special case $k=1$ is the main result of
\cite{brendleharnack}:

\begin{cor} (Equivalent to 
Brendle \cite[Proposition 9]{brendleharnack}.)
Suppose $g(t)$ is a complete Ricci flow on a manifold $\m$ for 
$t\in (0,T]$ with scalar curvature uniformly bounded from above.
If $\Rop(g(t))\in \cc_1(\m)$ for all $t\in (0,T]$
then $\Rop_\infty\in \cc_1(\check\m)$.
\end{cor}

Note that the main result \cite[Theorem 1]{brendleharnack} 
as stated is a little weaker than this corollary, but it is 
deduced from \cite[Proposition 9]{brendleharnack} which
is just as strong. 
The preserved cone
$\cc_1$ was discovered by Brendle and Schoen \cite{BS}
and originally described as the cone of curvature forms
of manifolds $(\n,g)$ for which $(\n,g)\times\R^2$ has
weakly positive isotropic curvature. The  $\cc_1$ formulation 
has the advantage emphasising that the cone is independent
of the background metric (a fact appreciated  by Brendle in \cite{brendleharnack}).

For intermediate $k$, Theorem \ref{harnackthm} can
be considered to be some sort of interpolation between
these two corollaries. One has weaker hypotheses than 
Hamilton's result, and stronger conclusions than 
Brendle's result (but weaker conclusions than 
Hamilton's result, and stronger hypotheses than 
Brendle's result).

In the bigger scheme of things,
Hamilton's result has been crucial in the study of 
3-manifolds using Ricci flow. There, the blow-ups of singularities
do satisfy the generally rather restrictive hypothesis of
weakly positive curvature form, although the most useful
conclusion is the trace Harnack result (Remark \ref{traceresult}).
Brendle's result, having the weakest hypothesis, 
currently holds most
promise for applications to the study of manifolds of
positive isotropic curvature \cite{micallef_moore} via
Ricci flow.

\brmk
Although we have stated Theorem \ref{harnackthm} for
the cones $\cc_k$, the same proof works for more general
Wilking cones (see Appendix \ref{wilking_appendix}) 
with the additional assumption that they
can be defined independently of a background metric.
Moreover, the heuristics suggest that the result holds for 
more general convex cones
 which are invariant under the Ricci flow, as we describe in
in the next section, although we do not currently know of any 
suitable cones $K$ other than those already dealt with
directly in this paper.
\ermk

\subsection{More general cones}
\label{actsect}

We want to imagine more general convex cones $K$ within 
$Sym_B^2(\Wedge^2V^*)$, with ${\rm dim} V = n + 1$, 
so that the assertion of
a Harnack inequality will be that $\Rop_\infty\in K$.
(Recall that a convex cone $K$ in a vector space is a subset 
such that
if $\la\geq 0$ and $a,b\in K$, then $\la a\in K$ and $a+b\in K$.)
Given the explanations of previous sections, we would like
$K$ to be an invariant cone under Ricci flow, and since the
metric $\check g_\infty$ is degenerate, it will be 
appropriate to ask that the cone
is $GL(V)$-invariant. To clarify both notions:

\begin{defn}
\label{Kdef}
A cone $K$ within the vector space $Sym_B^2(\Wedge^2V^*)$ 
of algebraic curvature tensors on an $m$-dimensional
real vector space $V$ is called 
$GL(V)$-invariant if for each $T\in K$
and linear map $A\in GL(V)$, the tensor
$T_A\in Sym_B^2(\Wedge^2V^*)$
defined by
$$T_A (v_1,v_2,v_3,v_4):=T (A v_1,A v_2,A v_3,A v_4)$$
also lies in $K$.

Such a $GL(V)$-invariant 
cone $K$ which is also closed and convex, 
is said to be invariant under Ricci flow if
when we define, for $T\in Sym_B^2(\Wedge^2V^*)$
and positive definite 
$g\in Sym^2 V$, the tensor
\beq
\label{Qdef}
Q(T,g)_{abcd}:=2g^{\al\ga}g^{\be\de}\left[
T_{a\al b\be}T_{c\ga d\de}
-T_{a\al b\be}T_{d\ga c\de}
+T_{a\al c\be}T_{b\ga d\de}
-T_{a\al d\be}T_{b\ga c\de}
\right]
\eeq
then for all $T\in \partial K$ and all (equivalently one) $g$, 
the tensor $Q(T,g)$ 
points into the interior of $K$ at $T$.  
That is, for any $\Ups\in (Sym^2_B(\Wedge^2 V^*))^*$
such that $\Ups(\tilde T-T) > 0$ implies
$\tilde T\notin K$, we have $\Ups(Q(T,g))\leq 0$.
Or, alternatively phrased (see \cite[Lemma 4.1]{ham4mfd} for a proof of the equivalence) the cone $K$ is preserved under
the ODE $\dot T = Q(T,g)$ (for any or all $g$).
\end{defn}

Note that we can talk about a curvature form/tensor $\Rop$
on a manifold $\n^m$ lying in a $GL(V)$-invariant cone $K$:
we ask that at each point $x\in \n$, after identifying
$T_x\n$ with $V$ via an arbitrary linear bijection,
we have $\Rop(x)\in K$.

The terminology of the definition 
is justified because Hamilton's ODE-PDE theorem (see Theorem 4.3 in \cite{ham4mfd})
tells us that if $g(t)$ is a Ricci flow on a closed manifold 
such that $\Rop(g(0))$ lies in a cone $K$ as in Definition 
\ref{Kdef}, then $\Rop(g(t))\in K$ for all later times $t$.

\begin{conj}
\label{generalconj}
Let $V$ be a vector space of dimension $n + 1$. 
Suppose $K$ is a closed $GL(V)$-invariant 
convex cone within $Sym^2_B(\Wedge^2 V^*)$
which contains all weakly positive definite curvature forms, and is invariant under the Ricci flow in the sense 
of Definition \ref{Kdef}.
Suppose further than $g(t)$ is a Ricci flow on a manifold $\m^n$ for 
$t\in (0,T]$ with uniformly bounded curvature, 
and $\Rop((\m,g(t))\times \R)\in K$ for all $(x, t) \in \check \m$. 
Then $\Rop_\infty\in K$.
\end{conj}

It would be particularly interesting in this context to find
such a $GL(V^{n+1})$-invariant cone containing the cone ${\cal C}(S_1)$
of Appendix \ref{wilking_appendix}, but contained in
the cone of curvature operators with 
weakly positive sectional curvature.

\section{Proof of the Harnack estimates.}
\label{proofsection}

In this section we give a rigorous proof of Theorem
\ref{harnackthm} following the ideas of the previous section,
and the earlier papers of Hamilton \cite{hamharnack} and
Brendle \cite{brendleharnack}.

\subsection{Applying the Canonical Expanding Soliton to derive 
equations satisfied by $\Rop_\infty$}

We have seen how the curvature of the metrics $\cg_N$ converges
to $\Rop_\infty$ as $N\to\infty$, and in this section we will
apply what we know about the curvatures of Ricci flows to 
derive information about this limit. More precisely,
we want to derive a parabolic-type equation 
\eqref{mainformula} which it satisfies.
We will also have to
consider one further object which behaves well in the 
same limit, namely the connection.

\begin{prop}
In the setting of Theorem \ref{CESthm}, the Levi-Civita 
connection $\check\grad$ of $\cg_N$
converges smoothly as $N\to\infty$ to a limit 
$\stackrel{\infty}{\grad}$.
More precisely, after choosing local coordinates 
$\{x_1,\ldots,x^n,t\}$ near some
point in $\check\m$, the Christoffel symbols with respect 
to those coordinates will converge smoothly.
At a point $(x,t)$ in $\check \m$, the Christoffel symbols
of $\stackrel{\infty}{\grad}$ are given in terms of the coefficients
of $g(t)$ and its curvature at the point $x$ by
$$
\stackrel{\infty}{\Ga}^{\raisebox{-1.5ex}{\scriptsize  i}}_{jk} =\Ga^i_{jk};\quad
\stackrel{\infty}{\Ga}^{\raisebox{-1.5ex}{\scriptsize  i}}_{j0}= - \left({R^i}_j+\frac{{\de^i}_j}{2 t}\right);\quad
\stackrel{\infty}{\Ga}^{\raisebox{-1.5ex}{\scriptsize  i}}_{00}=-\half g^{ij}\pl{R}{x^j};\quad
$$
$$
\stackrel{\infty}{\Ga}^{\raisebox{-1.5ex}{\scriptsize  $0$}}_{jk}=0;\quad
\stackrel{\infty}{\Ga}^{\raisebox{-1.5ex}{\scriptsize  $0$}}_{i0}=0;\quad
\stackrel{\infty}{\Ga}^{\raisebox{-1.5ex}{\scriptsize  $0$}}_{00}=-\frac{3}{2t}
$$
\end{prop}
The proof is immediate from Appendix \ref{computations}.

The above connection $\stackrel{\infty}{\grad}$ also arises
in Brendle \cite{brendleharnack}. Connections with similar Christoffel symbols can be found in \cite{CC1} (see also the modified version in \cite{CLN}) and \cite{ChK}.

This connection also induces a limiting Laplacian
\beq
\label{laptdef}
\lap_t:=
g^{ij}\stackrel{\infty}{\grad}_i\stackrel{\infty}{\grad}_j
\eeq
which can be applied to any tensor on $\check\m$.
Because $\stackrel{\infty}{\Ga}^{\raisebox{-1.5ex}{\scriptsize  $0$}}_{ij}=0$, we readily check
the following assertion.
\begin{lemma}
\label{laptlem}
If $f:\check\m\to\R$  then
$$\lap_t f=\lap_{g(t)}\left[f(\cdot,t)\right].$$ 
In particular, for any tensor $S$ on $\check\m$ and 
$F:\R\to\R$, we have $\lap_t (F(t) S) = F(t)\lap_t S$.
\end{lemma}

Recall now
(see for example \cite[Proposition 2.5.1]{RFnotes}) that
the curvature $\Rop$ of a Ricci flow $g(t)$ satisfies
the equation 
\beq
\label{Rmeq}
\pl{\Rop}{t}=\lap\Rop + F(\Rop,g)+Q(\Rop,g)
\eeq
where $Q$ was defined as in \eqref{Qdef} and $F$ is a map of
the same type defined by
\begin{align} \label{Fdef}
F(T,g)_{abcd} :&=
-g^{\al\be}g^{\ga\de}\left[
T_{\al bcd}T_{a\ga\be\de}+T_{a\al cd}T_{b\ga\be\de}
+T_{ab\al d}T_{c\ga\be\de}+T_{abc\al}T_{d\ga\be\de}\right] \nonumber
\\ & =-g^{\al\be}\left[
T_{\al bcd}T_{a\be}+ T_{a\al cd}T_{b\be}
+T_{ab\al d}T_{c\be}+T_{abc\al}T_{d\be}\right],
\end{align}
where $T_{a b} = g^{cd} T_{acbd}$.

As usual, we wish to apply this type of flow equation not to the
Ricci flow $g(t)$ under consideration, but to the
flow $G(s)$ of its Canonical Soliton $\cg_N$.
By Theorem \ref{CESthm}, Proposition \ref{psiprop} and
Proposition \ref{Geqprop} we have
$$\pl{G}{s}\simeq -2\Ric(G(s)),$$
Where $\simeq$ is used in the precise sense of Theorem
\ref{CESthm}, that is, the difference of the left-hand side 
and the right-hand side can be differentiated at will
using $\check\grad$, and multiplied by $N$, and will still
remain locally bounded as $N\to\infty$. 
Therefore, a glance at the derivation of \eqref{Rmeq}
shows that in fact
\beq
\label{Rmeq2}
\pl{\Rop_{G(s)}}{s}\simeq\lap_{G(s)}\Rop_{G(s)} + F(\Rop_{G(s)},G(s))
+Q(\Rop_{G(s)},G(s)).
\eeq

We may now set $s=1$ and take the limit $N\to\infty$.
Clearly 
$$F(\Rop_{G(1)},G(1))=F(\Rop_{\cg_N},\cg_N)\to F(\Rop_\infty,\check g_\infty)$$
and 
$$Q(\Rop_{G(1)},G(1))=Q(\Rop_{\cg_N},\cg_N)\to Q(\Rop_\infty,\check g_\infty)$$
as $N\to\infty$.
Meanwhile,
\begin{align*}
\pl{\Rop_{G(s)}}{s}&=\pl{}{s}\left[s\Rop_{\psi_s^*(\cg)}\right]
=\pl{}{s}\left[s\psi_s^*(\Rop_{\cg})\right]\\
&=\psi_s^*\left[\Rop_{\cg}
+s\lie_{\frac{1}{s}\check\grad(-\frac{N}{2t})}\Rop_{\cg}\right]\\
&\simeq\psi_s^*\left[\Rop_{\cg}
+\lie_{t\pl{}{t}}\Rop_{\cg}\right]\\
\end{align*}
by \eqref{Xformula}.
Evaluating at $s=1$ (where $\psi_1$ is the identity)
we find that
\beq
\pl{\Rop_{G(s)}}{s}\bigg|_{s=1}\to \Rop_\infty
+\lie_{t\pl{}{t}}\Rop_\infty
\eeq
as $N\to\infty$.
Finally, 
\begin{align*}
\lap_{G(1)}\Rop_{G(1)}&=\lap_{\cg}\Rop_{\cg} =
\cg^{ij}\check\grad_i\check\grad_j\Rop_{\cg} 
+ \cg^{00}\check\grad_0\check\grad_0\Rop_{\cg}
\end{align*}
and therefore 
\beq
\lap_{G(1)}\Rop_{G(1)}\to t\lap_t \Rop_\infty,
\eeq
where $\lap_t$ is from \eqref{laptdef}.

Combining all these formulae and passing to the limit 
in \eqref{Rmeq2}, we find the desired formula
\beq
\label{mainformula}
\lie_{t\pl{}{t}}\Rop_\infty=
t\lap_t \Rop_\infty + F(\Rop_\infty,\check g_\infty) + Q(\Rop_\infty,\check g_\infty)
-\Rop_\infty.
\eeq

\subsection{Being on the boundary of a $GL(V)$-invariant cone}

Suppose we have a convex cone $K$ in the space of algebraic curvature
forms which is 
$GL(V)$-invariant in the sense of Definition \ref{Kdef}.
In this section we show some useful identities which hold
whenever we are on the boundary of $K$, which will imply
that the second  
term on the right-hand side of 
\eqref{mainformula} points into the interior of $K$
when $\Rop_\infty\in \partial K$.  
If we assume that $K$ is invariant 
under Ricci flow (again in the sense of Definition \ref{Kdef})
then the third term on the right-hand side of 
\eqref{mainformula} points into the interior of $K$
when $\Rop_\infty\in \partial K$. 
Since $K$ is a cone, the same is true for the fourth term
on the right-hand side of \eqref{mainformula}.
One can speculate then that an appropriate ODE-PDE
theorem (similar to that proved by Hamilton in \cite{ham4mfd}) 
applied to $\Rop_\infty$ viewed as an appropriate 
time dependent section of a bundle over $\m$,
would tell us that
the PDE \eqref{mainformula} would leave $K$ invariant, and
that would give a proof of Conjecture \ref{generalconj}.

To avoid having to prove an appropriate ODE-PDE theorem 
(handling noncompact manifolds and without 
a fixed underlying bundle metric) we will take a more
direct route in this paper.

\begin{lemma}
\label{GLlemma}
If $\Th$ is an algebraic curvature tensor on the boundary 
of a $GL(V)$-invariant cone $K\subset Sym^2_B(\Wedge^2 V^*)$, 
then for any
$v_1,\ldots,v_4\in V$ and endomorphism $A\in V\otimes V^*$,
the   
tensor
$$(v_1,v_2,v_3,v_4)\mapsto
\Th(Av_1,v_2,v_3,v_4)+\Th(v_1,Av_2,v_3,v_4)+\Th(v_1,v_2,Av_3,v_4)+
\Th(v_1,v_2,v_3,Av_4)$$
points into the interior of $K$. Moreover, it points along the
boundary in the sense that its negation also points into the
interior.
\end{lemma}

\begin{proof}
If $\Th\in\partial K$, then $\Th_\al$ defined for $\al\in\R$ near $0$
by
$$\Th_\al(v_1,v_2,v_3,v_4):=
\Th((I+\al A)v_1,(I+\al A)v_2,(I+\al A)v_3,(I+\al A)v_4)$$
must also lie in $\partial K$ by virtue of the $GL(V)$-invariance.
Differentiating with respect
to $\al$ at $\al=0$ gives the first assertion, and the
rest follows by changing the sign of $A$.
\end{proof}

As a particular case, we have that $F(\Th,\check g_\infty)$ 
as defined in \eqref{Fdef} points along the boundary of $K$. Accordingly,
\begin{cor} \label{Cor_L3.3}
If the cone $K$ from Lemma \ref{GLlemma} is also a Wilking Cone ${\cal C}(S)$
(see Appendix \ref{wilking_appendix}) and $\dim V = n + 1$, then 
 $$F(\Th,\check g_\infty)(\om,\bar\om)=0 \qquad \text{whenever } \qquad  \Th(\om,\bar \om)=0 \quad \text{for } \om\in\Wedge^2 (V^\C).$$
\end{cor}

\subsection{Uniform control on $\Rop_\infty$}

As in Section \ref{CESasymp}, we make the change of
coordinates $\al=\ln t$, for $\al\in (-\infty,\ln T]$.
In these coordinates,
\beq
\label{Rmalpha}
\begin{aligned}
{\stackrel{\infty}{R}}_{ijkl} &  = e^{-\al} R_{ijkl} \\
{\stackrel{\infty}{R}}_{i\al j\al} & =  e^{\al}\(\lap R_{ij} + 2 R_{ikjl} R^{kl}   -R_i^k R_{jk} + \frac{R_{ij}}{2 t} - \frac1{2} \nabla_i \nabla_j R\)\\
{\stackrel{\infty}{R}}_{ij\al k} & = \nabla_i R_{jk} - \nabla_j R_{ik}
\end{aligned} 
\eeq

Define a metric $h$ on $\m\times (-\infty,\ln T]$ by
\beq
\label{hdef}
h:=g(e^\al)+d\al^2.
\eeq

Following the perturbation of the Harnack quadratic performed by Hamilton (cf. \cite{hamharnack} p. 239) and Lemma 7 in   \cite{brendleharnack} by Brendle, we will modify the algebraic curvature tensor $\Rop_\infty$  to another algebraic curvature tensor, which lies in the interior of the relevant cone.

\begin{prop} (cf. \cite[Lemma 7]{brendleharnack}.)
\label{firstgo}
Let $V$ be a vector space of dimension $n + 1$, and $K \subset Sym^2_B(\Wedge^2 V^*)$ a $GL(V)$-invariant convex cone   
containing all weakly positive definite curvature forms.
Suppose $g(t)$ is a Ricci flow on a manifold $\m^n$ for 
$t\in [0,T]$ with derivatives up to order two of its curvature 
uniformly bounded. 
If $\Rop((\m,g(t))\times\R)\in K$ for all $t  \in [0,T]$
then there exists $C_0<\infty$ such that  
$$\Rop_\infty + \eta h\odot h \in interior(K)$$
throughout $\check\m$ for $\eta\geq C_0$,
where $\odot$ refers to the Kulkarni-Nomizu product. 
\end{prop}

\begin{proof}
By hypothesis, $\Rop((\m,g(t))\times\R)\in K$ 
and therefore
$\frac{1}{t}\Rop((\m,g(t))\times\R)\in K$ for $t > 0$.
But 
$$\Rop_\infty-\frac{1}{t}\Rop((\m,g(t))\times\R)
=\Rop_\infty-e^{-\al}\Rop((\m,g(t))\times\R)$$ 
is bounded
when measured with respect to $h$ (for $\al\leq \ln T$)
by inspection of \eqref{Rmalpha}.

Therefore for large enough $\eta$, 
$\Rop_\infty-\frac{1}{t}\Rop((\m,g(t))\times\R)+ \eta h\odot h$
is positive definite (viewed as a bilinear form as usual)
and is thus in the interior of the cone $K$ by hypothesis.
The result follows.
\end{proof}

The above proposition could be considered a first step
to obtaining the result with $\eta=0$ in particular
cases -- i.e. obtaining a Harnack inequality. Moreover, it allows us to compare the quadratic quantities $Q$ and $F$, given by \eqref{Qdef} and \eqref{Fdef} respectively, at   $\Rop_\infty$
and at the perturbed $\Rop_\infty+\eta h\odot h$ whenever the latter hits the boundary of $K$ as follows:
\begin{cor} (cf. Brendle \cite{brendleharnack}.)
\label{quadcor}
In the setting of Proposition \ref{firstgo}, there
exists $M_2\leq \infty$ such that the following is true.
If at some point in $\check\m$ we have
$\Rop_\infty+\eta h\odot h \in \partial K$ for some $\eta>0$,
then at that point we have
$$Q(\Rop_\infty+\eta h\odot h, \check g_\infty)-Q(\Rop_\infty, \check g_\infty)
\leq M_2\eta t\, h\odot h$$
and 
$$F(\Rop_\infty+\eta h\odot h, \check g_\infty)-F(\Rop_\infty, \check g_\infty)
\leq M_2\eta t\, h\odot h.$$
\end{cor}

The corollary follows from Proposition \ref{firstgo} because
that proposition ensures that $\eta$ cannot be too large
(or $\Rop_\infty+\eta h\odot h$ could not be on the boundary of $K$)
and we may then use the fact that $Q$ and $F$ are both quadratic.

\subsection{Perturbing $\Rop_\infty$}

As is typical with a maximum principle proof, we must
consider the evolution of a slightly perturbed version of
$\Rop_\infty$ in order to proceed via a contradiction argument.
(In this respect we also follow \cite{hamharnack} and
\cite{brendleharnack}.)

First we note that $h$ defined in \eqref{hdef} satisfies
$$\lie_{t\pl{}{t}}h=\lie_{\pl{}{\al}}h=t\pl{g}{t}=-2t\Ric(g(t)),$$ 
and a short computation yields 
\beq
\label{hfirstderiv}
\stackrel{\hspace*{-0.1cm} \infty}{\grad_i}h
=\(tR_{ij}+\half g_{ij}\)\(d\al\tensor dx^j + dx^j\tensor d\al\),
\eeq
and using \eqref{laptdef},
$$\lap_t h = \frac{t}{2}\(dR\tensor d\al+ d\al\tensor dR\)
+\half|2t\Ric+ g|^2 d\al\tensor d\al.$$
Consequently, if we have a uniform curvature bound,
and a uniform bound on the first derivative of $R$ (e.g. in the setting
of Proposition \ref{firstgo}) then
\beq
\label{hcontrol}
|\lie_{t\pl{}{t}}h|_h\leq Ct;\qquad |\lap_t h|_h\leq C.
\eeq
We now require a `defining function' $\vph$ as in the work
of Hamilton and Brendle.

\begin{lemma} (Hamilton \cite{hamharnack} and
Brendle \cite{brendleharnack}.)
Let $(\m^n, g(t))$, $t \in [0,T]$, with $T < \infty$, be a complete solution of the Ricci flow such that $|\Rop(g(t))|$ and $|\nabla \Ric(g(t))|$ are bounded on $\m \times [0,T]$ by 
$C_0 \geq 0$. Then there exists a function 
$\varphi \in C^\infty(\m)$ and 
$C = C(n, C_0, T)<\infty$
with the properties:
\begin{compactenum}
\item
$\varphi(x) \to \infty$ as $x \to \infty$.
\item
$\varphi(x) \geq 1$ for all $x \in \mathcal M$.
\item
$\sup_{\mathcal M \times [0, T]} |\nabla \varphi|_{g(t)} \leq C$.
\item
$\sup_{\mathcal M \times [0, T]} |\Delta_{g(t)} \varphi| \leq C$.
\end{compactenum}
\end{lemma}

Using $\vph$, we define 
\beq
\label{Upsdef}
\Ups:=\vph\, h\odot h,
\eeq
and by \eqref{hcontrol} and \eqref{hfirstderiv}, 
$$|\lie_{t\pl{}{t}}\Ups|_h\leq Ct;\qquad |\lap_t \Ups|_h\leq C.$$
Therefore there exists $M_1>0$ such that
\beq
\label{Upscontrol}
\lie_{t\pl{}{t}}\Ups-t\lap_t \Ups+tM_1\Ups\geq 0.
\eeq

\begin{lemma}
\label{lambdalemma}
In the setting of Proposition \ref{firstgo}, there exists
$\la\in (0,\infty)$ such that for all $\ep>0$, 
if we define a tensor $\Th$ on $\check\m$ by
$$\Th:=t\Rop_\infty+\ep e^{\la t}\Ups$$
then 
at any point $(x,t)\in\check\m$ 
where $\Th$ is on the boundary of $K$, we have 
\beq
\label{Theq}
\lie_{t\pl{}{t}}\Th-t\lap_t \Th >
\frac{1}{t}Q(\Th, \check g_\infty)+\frac{1}{t}F(\Th, \check g_\infty).
\eeq
Moreover, there exist $\de>0$ and $\Om\subset\subset\m$ 
so that at any point $(x,t)\in\check\m$ with $t\leq \de$ 
or $x\notin\Om$, we have $\Th(x,t)\in interior(K)$.
\end{lemma}

Note that for a bilinear form $B$ on a vector space $V$,
we write $B>0$ if $B(v,v)>0$ for all \emph{nonzero} $v\in V$.

\begin{proof}
For initially arbitrary $\la>0$ and $\ep>0$, take the
corresponding $\Th$
and suppose that $\Th(x,t)\in\partial K$.
By Corollary \ref{quadcor}, we have
\beq
\label{Qconseq}
\frac{1}{t}Q(\Th, \check g_\infty)-t Q(\Rop_\infty, \check g_\infty)
\leq M_2 \ep t e^{\la t} \Ups,
\eeq
and the same for $F$.
By computing
$$\lie_{t\pl{}{t}}\Th=t\Rop_\infty+t\lie_{t\pl{}{t}}\Rop_\infty
+\la t\ep e^{\la t}\Ups+ \ep e^{\la t}\lie_{t\pl{}{t}}\Ups,$$
and (using Lemma \ref{laptlem})
$$\lap_t \Th=t\lap_t\Rop_\infty + \ep e^{\la t}\lap_t\Ups$$
we find, using \eqref{mainformula}, \eqref{Upscontrol}
and \eqref{Qconseq}, that
\begin{align*}
\lie_{t\pl{}{t}}\Th-t\lap_t \Th&=
\la t\ep e^{\la t}\Ups+ \ep e^{\la t}\(\lie_{t\pl{}{t}}\Ups
-t\lap_t\Ups\)
+t F(\Rop_\infty,\check g_\infty) + tQ(\Rop_\infty,\check g_\infty)\\
&\geq \ep t e^{\la t}\Ups \(\la - M_1 -2M_2\)
+\frac{1}{t}Q(\Th, \check g_\infty)+\frac{1}{t}F(\Th, \check g_\infty)
\end{align*}
Therefore, if $\la$ is chosen large enough, we have
\eqref{Theq} as desired.

Finally, since
$$\Th=t\(\Rop_\infty+\frac{\ep e^{\la t}\vph}{t}h\odot h\),$$
by Proposition \ref{firstgo} (taking $C_0$ 
from that proposition) if $t\leq \frac{\ep}{C_0}=:\de$ or
$\vph(x)\geq \frac{C_0 T}{\ep}$
(i.e. $x\notin \vph^{-1}\([1,\frac{C_0 T}{\ep})\)=:\Om$)
then $\Th(x,t)\in interior(K)$, and we deduce the final part of the lemma.
\end{proof}

Notice that the last claim of the previous lemma effectively gives us the  initial condition for the
maximum principle argument we will use to prove Theorem \ref{harnackthm}; in fact, it ensures that $\Th$ (which should be regarded as a slight  perturbation of $\Rop_\infty$)
is in the cone for small enough $t$. Furthermore, it says nothing
bad can happen at spatial infinity, i.e. if we were to
fall out of the cone, it would happen within some compact region.

\subsection{Proof of Theorem \ref{harnackthm}}

The Ricci flow of the theorem has uniformly bounded curvature
by Remark \ref{traceresult}, and so by Shi's derivative estimates,
for any $\de>0$ we may assume that all derivatives of the
curvature are bounded for $t\in [\de,T]$ (with bounds
depending on $\de$).
Therefore, it suffices to prove the result assuming that
the Ricci flow exists on a closed time interval $[0,T]$ and
that all derivatives of the curvature are bounded on $[0,T]$,
since we could apply that result to $g(\de+t)$ for $t\in [0,T-\de]$
and then let $\de\downto 0$.

We will apply the results we have established so far in
Section \ref{proofsection} in the case that 
$K$ is the closed cone $\cc(S_k)$ (as discussed in Appendix
\ref{wilking_appendix}) on an $(n+1)$-dimensional vector space.
We begin by choosing $\la$ as in Lemma \ref{lambdalemma}.
For each $\ep>0$, we can then define the corresponding $\Th$.
Since our goal is to show that $\Rop_\infty \in \cc(S_k)$,
all we have to do is to show that $\Th$ lies in the cone 
$\cc(S_k)$ everywhere, for this arbitrarily small $\ep$.

By the second part of Lemma \ref{lambdalemma}, we know that 
$\Th$ lies in the interior of the 
cone $\cc(S_k)$ for small $t$ and
also,
outside $\Om$ for all time. 

Suppose, contrary to what we wish to establish,
that at some point $\Th$ fails
to lie in the cone.
Then we can pick a point $(x_1,t_1)$ with $t_1>0$ such that $\Th(x_1,t_1)$ is on
the boundary of the cone, and without loss of generality,
we may assume that $t_1$ is the least possible time at which we could
find such a point.

Since $\Th(x_1,t_1)$ is on the boundary of the cone $\cc(S_k)$, 
we can pick a nonzero element
$\om\in S_k\subset \Wedge^2 ({T}_{(x_1,t_1)}^\C\check\m)$ such
that $\Th (\om,\bar\om)=0$.
By Corollary \ref{Cor_L3.3}, 
since the cone is $GL(n + 1,\R)$-invariant, we have, at $(x_1,t_1)$,
$$F(\Th,\check g_\infty)(\om,\bar\om)=0$$
and by Theorem \ref{wilkingthm} 
we have 
$$Q(\Th,\check g_\infty)(\om,\bar\om)\geq 0,$$
so by Lemma \ref{lambdalemma} 
\beq
\label{oppositeineq}
\left(\lie_{t\pl{}{t}}\Th-t\lap_t \Th \right)(\om,\bar\om)> 0.
\eeq

We now extend $\om$ to a neighbourhood of $x$ in $\m$ by 
parallel translation along radial geodesics using
$\stackrel{\infty}{\grad}$, and then
extend $\om$ in time to make it constant in the sense
that 
\beq
\label{omconst}
\lie_{t\pl{}{t}}\om=0.
\eeq
By construction, 
\beq
\label{lapzero}
\lap_t\om=0
\eeq
at $(x_1,t_1)$.
Since parallel translation of $\om$ preserves its rank,
we have 
$rank(\om)\leq k$ 
in a neighbourhood of $(x_1,t_1)$.

Next we define a function $f$ in that neighbourhood of $(x_1,t_1)$ by 
$$f:=\Th (\om,\bar\om).$$
By definition of $f$, we have $f\geq 0$ near $x_1$
up to time $t_1$, and $f(x_1,t_1)=0$, and therefore at $(x_1,t_1)$
we have $\pl{f}{t}\leq 0$ and $\lap_{g(t)} f\geq 0$, so that
\beq
\label{oneway}
\pl{f}{t}-\lap_{g(t)} f\leq 0.
\eeq
On the other hand, we can compute at $(x_1,t_1)$
$$t\pl{f}{t}=\lie_{t\pl{}{t}}\Th (\om,\bar\om) + 
\Th(\lie_{t\pl{}{t}}\om,\bar\om)+ 
\Th(\om,\lie_{t\pl{}{t}}\bar\om)
=\lie_{t\pl{}{t}}\Th (\om,\bar\om),$$
by \eqref{omconst}, and 
$$\lap_t f = (\lap_t\Th)(\om,\bar\om)
+ \Th (\lap_t\om,\bar\om)
+ \Th (\om,\lap_t\bar\om)
= (\lap_t\Th)(\om,\bar\om)$$
by \eqref{lapzero}. 
Therefore by Lemma \ref{laptlem} and \eqref{oppositeineq}
we have at $(x_1,t_1)$
$$\pl{f}{t}-\lap_{g(t)} f
=\frac{1}{t}\(\lie_{t\pl{}{t}}\Th - t\lap_t\Th\)(\om,\bar\om)>0,$$
contradicting \eqref{oneway}, and completing the argument.

\appendix

\section{Appendix - computations}
\label{computations}

In this appendix we give the formulae for the connection
and curvature associated to the metric $\cg$ from
Theorem \ref{CESthm}. We also compute the Hessian
of $-\frac{N}{2 t}$, thus justifying the assertions
of Theorem \ref{CESthm}.

\begin{prop}
In the setting of Theorem \ref{CESthm}, if $\Ga^i_{jk}$
are the Christoffel symbols of $g(t)$ at some point
$x\in\m$, then the Christoffel symbols of $\check g$
at $(x, t)$ are given by
$$
\check \Ga^i_{jk}=\Ga^i_{jk};\quad
\check \Ga^i_{j0}= - \left({R^i}_j+\frac{{\de^i}_j}{2 t}\right);\quad
\check \Ga^i_{00}=-\half g^{ij}\pl{R}{x^j};\quad
$$
$$
\check \Ga^0_{jk}=\check g_{00}^{-1}
\left(\frac{R_{jk}}{t} + \frac{g_{jk}}{2 t^2}\right);\quad
\check \Ga^0_{i0}=\frac{1}{2 t} \check g_{00}^{-1}\pl{R}{x^i};\quad
\check \Ga^0_{00}=-\frac{3}{2t}+\frac{\check g_{00}^{-1}}{2t} 
\left[\frac{2R}{t}+R_t +\frac{n}{2t^2}\right]
$$
for $N$ sufficiently large that $\cg$ is a Riemannian metric.
\end{prop}

This proposition follows from the definition of the
Christoffel symbols
$$\check\Ga^a_{bc}:=\half \check g^{ad}\left(
\pl{\check g_{cd}}{x^b}+\pl{\check g_{bd}}{x^c}-\pl{\check g_{bc}}{x^d}
\right),$$
where $a, b, c, d$ are arbitrary indices,
and the equation of Ricci flow. Using the
standard formula for the coefficients of the curvature tensor
$${{{\check R}_{abc}}}\/ ^d
=\pl{\check\Ga^d_{ac}}{x^b} -\pl{\check\Ga^d_{bc}}{x^a}
+\check\Ga^e_{ac}\check\Ga^d_{be}-\check\Ga^e_{bc}\check\Ga^d_{ae},$$
(where our sign convention is that $R_{1212}\geq 0$ on a 
positively curved manifold)
and the equation for the evolution of $\Ric$ under Ricci flow
$$\pl{R_{i}^j}{t}= \lap R_i^j + 2{R^j}_{nim}R^{mn},$$
(see for example \cite{RFnotes})
we can then verify the formulae of the following proposition.

\begin{prop}
In the setting of Theorem \ref{CESthm}, the coefficients of
the curvature tensor $\Rop$ of $\cg$ at $(x,t)$ are given by
\begin{eqnarray*} 
\check{R}_{ijkl} &  = &
\frac1{t} R_{ijkl} -  \frac{\check g^{00}}{2t^2} \left[
\(\Ric + \frac{g}{2t}\)
\odot
\(\Ric + \frac{g}{2t}\)
\right]_{ijkl}\\ 
&= &
\frac1{t} R_{ijkl} -  \frac{\check g^{00}}{t^2} \left[R_{ik} R_{jl}  - R_{il} R_{jk} + \frac1{2t} \(R_{ik} g_{jl}  + R_{jl} g_{ik} - R_{il} g_{jk} - R_{jk} g_{il}\) \right.
\\ & & \hspace*{2cm}  \left. \quad + \frac1{4 t^2}\(g_{ik} g_{jl} - g_{il} g_{jk}\)  \right] 
\\
\check R_{i0j0} & = &  \frac1{t}\(\lap R_{ij} + 2 R_{ikjl} R^{kl}   -R_i^k R_{jk} + \frac{R_{ij}}{2 t} - \frac1{2} \nabla_i \nabla_j R\)
\\ & & + \frac{\check g^{00}}{2 t^2}\left[\frac1{2} \nabla_i R \nabla_j R - \(\frac{2 R}{t} + \partial_t R + \frac{n}{2 t^2}\) \(R_{ij} + \frac{g_{ij}}{2t}\)\right]
\\
\check{R}_{ij0k} & = & \frac1{t}\(\nabla_i R_{jk} - \nabla_j R_{ik}\)  +\frac{\check g^{00}}{2t^2}\left\{ \(R_{ik} + \frac{g_{ik}}{2t}\) \nabla_j  R - \(R_{jk} + \frac{g_{jk}}{2t}\) \nabla_i R\right\}
 \end{eqnarray*} 
for $N$ sufficiently large that $\cg$ is a Riemannian metric.
\end{prop}

By taking the appropriate trace to give Ricci curvatures,
and by using 
the formula for the coefficients of $\Hess_{\check g}(f)$
$$\check\grad^2_{ab}(f)=\plndd{f}{x^a}{x^b}-\pl{f}{x^c}\check\Ga^c_{ab},$$
the equation for the evolution of $R$
$$R_t - \lap R - 2|\Ric|^2=0,$$
(see for example \cite[Proposition 2.5.4]{RFnotes})
and the contracted second Bianchi identity
$$\grad_i {R^i}_j=\half \grad_j R,$$
we find:

\begin{prop}
\label{ric_hess_prop}
In the setting of Theorem \ref{CESthm}, the coefficients of
the Ricci tensor $\Ric$ of $\cg$ at $(x,t)$ are given by\begin{eqnarray*}
\check R_{ij} &   = &R_{ij} +  \frac{\cg^{00}}{t} \left[M_{ij}-\(R_{ij} + \frac{g_{ij}}{2 t}\) \(R + \frac{n}{2 t}\) + \(R_{i}^k + \frac{\delta_i^k}{2 t}\) \(R_{jk} + \frac{g_{jk}}{2 t}\) \right]
 \\ & & + \frac{(\cg^{00})^2}{2 t^2}\left[\frac1{2} \nabla_i R \nabla_j R - \(\frac{2 R}{t} + \partial_t R + \frac{n}{2 t^2}\) \(R_{ij} + \frac{g_{ij}}{2t}\)\right]
\\
\check R_{i0}  & = & \frac{\nabla_i R}{2} +\frac{\check g^{00}}{2t}\left[ \(R_{i}^j + \frac{\delta_{i}^j}{2t}\) \nabla_j  R - \(R + \frac{n}{2t}\) \nabla_i R\right]
\\
\check R_{00} & = & \frac{\partial_t R}{2}   + \frac{R}{2 t}
+ \frac{\check g^{00}}{2 t}\left[\frac1{2} |\nabla R|^2 - \(\frac{2 R}{t} + \partial_t R + \frac{n}{2 t^2}\) \(R + \frac{n}{2t}\)\right]
\end{eqnarray*}
where
$$M_{ij}:=\lap R_{ij} + 2 R_{ikjl} R^{kl}   -R_i^k R_{jk} + \frac{R_{ij}}{2 t} - \frac1{2} \nabla_i \nabla_j R.$$
Moreover, 
we have
\begin{eqnarray*}
\check \nabla_{ij}^2 \(- \frac{N}{2 t}\) &= & - \(R_{ij} + \frac1{2t} g_{ij}\) + \cg^{00} \(\frac{R}{t} + \frac{n}{2 t^2}\)\(R_{ij} + \frac1{2t} g_{ij}\)
\\
\check \nabla_{i0}^2 \(- \frac{N}{2 t}\) &= &  - \frac{\nabla_i R}{2} + \frac{\cg^{00}}{2} \nabla_i R \(\frac{R}{t} + \frac{n}{2 t^2}\)
\\
\check \nabla_{00}^2 \(- \frac{N}{2 t}\) &= &  - \frac{1}{2} \check g_{00}  -\frac{\partial_t R}{2} - \frac{R}{2t}  + \frac{\check g^{00}}{2 t}\(R + \frac{n}{2 t}\)\(\partial_t R  + \frac{n}{2t^2} + \frac{2 R}{t}\)   
\end{eqnarray*}
\end{prop}

By combining the formulae of Proposition \ref{ric_hess_prop}
and the definition of $\check g$, we deduce:
\begin{align*}
\check R_{ij} + \check \nabla_{ij}^2 \(- \frac{N}{2 t}\) 
+\half \cg_{ij}
& =  \frac{\cg^{00}}{t} \(\lap R_{ij} + 2 R_{ikjl} R^{kl}   + 
\frac{3R_{ij}}{2t} - \frac1{2} \nabla_i \nabla_j R + \frac{g_{ij}}{4 t^2}\)
 \\ & \quad + \frac{(\cg^{00})^2}{2 t^2}\left[\frac1{2} \nabla_i R \nabla_j R - \(\frac{2 R}{t} + \partial_t R + \frac{n}{2 t^2}\) \(R_{ij} + \frac{g_{ij}}{2t}\)\right]\\
\check R_{i0} + \check \nabla_{i0}^2 \(- \frac{N}{2 t}\) 
+\half \cg_{i0}
&= \frac{\check g^{00}}{2t} \(R_{i}^j + \frac{\delta_{i}^j}{2t}\) \nabla_j R\\
\check R_{00} + \check \nabla_{00}^2 \(- \frac{N}{2 t}\) 
+ \frac{1}{2} \check g_{00}
&= 
\frac{\cg^{00}}{4 t} |\nabla R|^2
\end{align*}
and from these formulae it is easy to deduce Theorem 
\ref{CESthm}.

\section{Appendix - Wilking's cones}
\label{wilking_appendix}

Let V be a vector space of dimension $m$. 
 Given an algebraic curvature form  
$\Rop\in Sym_B^2(\Wedge^2V^*)$, and an inner product
$g$ on $V$, we wish to consider $Q(\Rop,g)$ defined
as in \eqref{Qdef}  which is a vector in
$Sym_B^2(\Wedge^2V^*)$.

In this section, the goal is to find closed convex cones within the
space $Sym^2_B(\Wedge^2 V^*)$ (whose definition generally
can use the inner product $g$) which are invariant under the ODE
\beq
\label{RODE}\frac{d\Rop}{d t}=Q(\Rop,g)
\eeq
(for fixed $g$) and invariant under the action of $O(V)$ 
(with respect to $g$).
Thus we have a slightly more general setting than
in Definition \ref{Kdef}, and a combination of the
so-called Uhlenbeck trick and Hamilton's ODE-PDE theorem (cf. \cite{ham4mfd})
says that these cones are invariant under Ricci flow
on closed manifolds.

We now describe a special case of Wilking's unpublished method 
\cite{wilking} for 
finding a large number of such cones.

First, consider $V^\C$, the complexification of $V$, and extend each curvature
operator/form  complex linearly to 
act on $\Wedge^2 (V^\C)$
as in \cite{micallef_moore}. 
(We will make such extensions implicitly in what follows.)

Using the inner product $g$, we can pick an orthonormal basis
$\{e_i\}$ for $V$, and consider the corresponding 
action of $SO(m,\C)$ on $V^\C$.
Explicitly, an element $\sum v_i e_i\in V^\C$ (with $v_i\in\C$)
would be 
mapped by $A\in SO(m,\C)$ to the element $\sum_{i,j}A_{ij}v_j e_i$.
This action then extends naturally to $\Wedge^2 (V^\C)$.

\begin{thm} (Special case of Wilking \cite{wilking}.)
\label{wilkingthm}
Suppose $S\subset \Wedge^2 (V^\C)$ is a subset 
(not necessarily a linear subspace)
which is invariant under the action of $SO(m,\C)$.
Then the convex cone
$$\cc(S):=\{\Rop\in Sym_B^2(\Wedge^2 V^*)\ |\ 
\Rop(\om,\bar \om)\geq 0 \text{ for all }
\om\in S\}$$
is invariant under the ODE \eqref{RODE}. Equivalently,
if $\Rop\in\partial \cc(S)$ and we pick a nonzero $\om\in S$ such
that $\Rop(\om,\bar \om)=0$, then $Q(\Rop,g)(\om,\bar\om)\geq 0$.
\end{thm}

By taking different choices of $S$, Wilking recovered all the 
famous invariant curvature cones such as weakly positive 
isotropic curvature (WPIC) and its variants. 

Note that the invariance of $S$ under the action of
$SO(m,\C)$ will depend on the inner product $g$, but 
not on the specific orthonormal basis we chose.
For the purposes of this paper, we are interested to find
cones in $Sym_B^2(\Wedge^2 V^*)$ which do not depend on $g$,
and are therefore interested to find sets $S$ which
are invariant under the natural action of $GL(m,\R)$
in addition to $SO(m,\C)$.

Given an element $\om\in\Wedge^2 (V^\C)$, we define 
$rank(\om)\in\{0,1,2,\ldots,[\frac{m}{2}]\}$ to be the
least number of simple elements $u\wedge v$ we need sum to 
obtain $\om$. We then define
$$S_k=\{\om\in \Wedge^2 (V^\C)\ |\ rank(\om)\leq k\},$$
and notice that it is both $GL(m,\R)$ and $SO(m,\C)$
invariant.
By Theorem \ref{wilkingthm}, the 
curvature cones $\cc(S_k)$, which are
invariant under the action of $GL(m,\R)$,
are also invariant under the ODE \eqref{RODE}
for one, and thus all $g$ (even degenerate \emph{weakly}
positive definite $g\in Sym^2 V$, by approximation).

As discussed in Section \ref{resultssect}, 
one can check that $\cc(S_1)$ is equal to a cone of curvature 
forms introduced by Brendle and Schoen \cite{BS}. 
Meanwhile, since $\Wedge^2 V\subset S_{[\frac{m}{2}]}$, 
the cone $\cc(S_{[\frac{m}{2}]})$ consists of all  weakly positive 
definite curvature forms, i.e. the cone of curvatures
of manifolds of weakly positive curvature operator.

\subsection{Proof of Wilking's result}

For completeness, we give Wilking's proof of Theorem 
\ref{wilkingthm}. (Any inefficiencies or inaccuracies are due to us.)
Using the basis $\{e_i\}$ for $V$, 
we have a (linear) isomorphism between 
$\Wedge^2 V$ and $\frak{so}(m,\R)$ -- and hence between
$\Wedge^2 (V^\C)$ and $\frak{so}(m,\C)$ -- determined by 
$$e_i\wedge e_j \leftrightarrow [e_i\wedge e_j]_{kl}:=
\de_{ik}\de_{jl}-\de_{il}\de_{jk}.$$
We will thus see curvature operators as bilinear
forms, or operators, on $\frak{so}(m,\R)$ or $\frak{so}(m,\C)$ 
without change of notation.
The action of $A\in SO(m,\C)$ on $\Wedge^2 (V^\C)$ then
corresponds to the adjoint representation of $SO(m,\C)$,
$Ad_{A}:\frak{so}(m,\C) \to \frak{so}(m,\C)$  
which can be written
in terms of matrix multiplication as $Ad_A(v)=AvA^{-1}$,
so we may view $S$ as a subset of $\frak{so}(m,\C)$ invariant
under $Ad$.

We will also need the adjoint representation 
$ad_X: \frak{so}(m,\R) \to \frak{so}(m,\R)$ of $\frak{so}(m,\R)$ (and the
same for $\frak{so}(m,\C)$):
\beq
\label{diff}
ad_X Y=\frac{d}{dt}Ad_{\exp(tX)}Y\bigg|_{t=0}=[X,Y],
\eeq
and the formula
\beq
\label{diff2}
\frac{d^2}{dt^2}Ad_{\exp(tX)}Y\bigg|_{t=0}=ad_X ad_X Y.
\eeq

The natural inner product on 
$\frak{so}(m,\R)$ is invariant under the adjoint representation
of $SO(m,\R)$: 
for every $A\in SO(m,\R)$, 
\beq
\label{inv}
\langle X,Y \rangle = \langle Ad_A X,Ad_A Y \rangle,
\eeq
for all $X,Y\in \frak{so}(m,\R)$.
By setting $A=\exp(tZ)$
in \eqref{inv}, for $Z\in \frak{so}(m,\R)$, 
and differentiating with respect to $t$
at $t=0$ using \eqref{diff}, we know that
\beq
\label{adid}
0=\langle ad_Z X,Y\rangle + \langle X,ad_Z Y\rangle.
\eeq
The formula then extends 
to $X,Y,Z\in \frak{so}(m,\C)$ with $ad$ the adjoint representation
of $\frak{so}(n,\C)$.

Recall (Hamilton \cite{ham4mfd}, B\"ohm-Wilking \cite{BW})
that one can write 
$$Q(\Rop,g)=\Rop^2+\Rop^\#,$$
where
$$\Rop^\#(X,Y):=-\tr\left(ad_X\circ \Rop\circ ad_Y\circ \Rop\right),$$
the trace taking place over any orthonormal basis for $\frak{so}(m,\R)$,
or indeed over any unitary basis for $\frak{so}(m,\C)$. 
Therefore, since $\Rop^2$ is weakly positive definite,
it suffices to prove that 
$$\Rop^\#(v,\bar v)\geq 0$$
(extending bilinear forms complex linearly as usual)
whenever $v\in S\subset \frak{so}(m,\C)$ satisfies $\Rop(v,\bar v)=0$,
or equivalently 
$$-\tr\left(ad_{\bar v}\circ \Rop\circ ad_v\circ \Rop\right)\geq 0.$$

Because $\Rop(v,\bar v)=0$ and $\Rop(w,\bar w)\geq 0$ for all 
$w\in S$ (and $S$ is invariant under $Ad$) we have
$$0\leq \Rop(Ad_{\exp(tx)}v,Ad_{\exp(t\bar x)}\bar v)$$
for all $x\in \frak{so}(m,\C)$, and differentiating twice with respect
to $t$ using \eqref{diff} and \eqref{diff2} we find that
$$0\leq \Rop(ad_x ad_x v, \bar v)+ 
\Rop(v, ad_{\bar x} ad_{\bar x} \bar v)
+2 \Rop(ad_x v,ad_{\bar x}\bar v).$$
If we change $x$ to $ix$ to get a new inequality and then add the two
inequalities, we get $0\leq \Rop(ad_x v,ad_{\bar x}\bar v)$,
and hence
\beq
\label{mainstep}
\Rop(ad_v x,ad_{\bar v}\bar x)\geq 0,
\eeq
for all $x\in \frak{so}(m,\C)$.
Note we have shown that by virtue of being on the boundary
of $\cc(S)$, $\Rop$ is positive not only on $S$, but also on the 
entire image of $ad_v$.

Moreover, \eqref{mainstep} and \eqref{adid} imply
$$0\leq 
\langle \Rop( ad_v x),ad_{\bar v}\bar x\rangle
=-\langle ad_{\bar v}\circ \Rop\circ ad_v x,\bar x\rangle,$$
for all $x\in \frak{so}(m,\C)$ and we find that 
$-ad_{\bar v}\circ \Rop\circ ad_v \geq 0$, or equivalently
$$-ad_{v}\circ \Rop\circ ad_{\bar v} \geq 0,$$
i.e. a positive definite Hermitian operator on $\frak{so}(m,\C)$.

Thus we can diagonalise $-ad_{v}\circ \Rop\circ ad_{\bar v}$ on
$\frak{so}(m,\C)$, finding a unitary basis $\{\om_i\}$
(i.e. $\langle \om_i,\bar \om_i\rangle = \de_{ij}$)
such that $-ad_{v}\circ \Rop\circ ad_{\bar v} (\om_i)=  \la_i \om_i$,
with $\la_i\geq 0$.
For any $i$ such that $\la_i>0$, we therefore find that
$\om_i$ is in the image of $ad_v$, and thus
$\Rop(\om_i,\bar \om_i)\geq 0$ by \eqref{mainstep}. 
Consequently,
\beq
\begin{aligned}
-\tr\left(ad_{\bar v}\circ \Rop\circ ad_v\circ \Rop\right)
&=-\tr\left( \Rop\circ ad_v\circ \Rop \circ ad_{\bar v}\right)\\
&=\sum_i 
\langle \Rop\circ ad_v\circ \Rop \circ ad_{\bar v} (\om_i),\bar \om_i\rangle\\
&=\sum_i \la_i \langle \Rop(\om_i),\bar \om_i\rangle\\
&=\sum_i \la_i \Rop(\om_i,\bar \om_i)\\
&\geq 0.
\end{aligned}
\eeq

\brmk
Viewing $S$ as a subset of $\frak{so}(m,\C)$ makes it easier to
describe some further interesting examples (as observed
by Wilking). The cone $\cc(S)$ corresponding
to $S=\{X\ |\ rank(X)=2;\ X^2=0\}$, for example, is
the cone of weakly positive isotropic curvature (WPIC) 
operators.
\ermk

{\sc mathematics institute, university of warwick, coventry, CV4 7AL,
uk}\\
\url{http://www.warwick.ac.uk/staff/E.Cabezas-Rivas}\\
\url{http://www.warwick.ac.uk/~maseq}

\begin{thebibliography}{99}

\bibitem{BS} S. Brendle and R. Schoen, \emph{Manifolds with $1/4$-pinched curvature are space forms.\/} J. Amer. Math. Soc. {\bf 22}
(2009) 287--307.

\bibitem{brendleharnack} S.~Brendle, \emph{A generalization 
of Hamilton's differential Harnack inequality for the Ricci flow.\/}
J. Differential Geom. {\bf 82} (2009), 207--227.

\bibitem{BW}  C. B\"ohm and B. Wilking, \emph{Manifolds with positive curvature operators are space forms.} Ann. of Math. (2) {\bf 167} (2008), no. 3, 1079--1097.

\bibitem{CT1} E.~Cabezas-Rivas and P.M.~Topping, 
\emph{The Canonical Shrinking Soliton associated to a Ricci flow.\/}
Preprint (2008) \url{http://www.warwick.ac.uk/~maseq}

\bibitem{CC1} B. Chow and S.-C. Chu, \emph{A geometric 
interpretation of Hamilton's Harnack inequality for the
Ricci flow.\/} Math. Res. Lett. {\bf 2} (1995) 701--718.


\bibitem{ChK} B. Chow and  D. Knopf, \emph{New Li-Yau-Hamilton
inequalities for the Ricci Flow via the space-time approach.\/}
J. Differential Geom. {\bf 60} (2002), 1--54.

\bibitem{CLN} B. Chow, P. Lu and L. Ni, Hamilton's Ricci flow.
G.S.M. {\bf 77} A.M.S. (2006).

\bibitem{ham3D} R. S. Hamilton, \emph{Three-manifolds with positive
Ricci curvature.\/} J. Differential Geom. {\bf 17} (1982) 255--306.

\bibitem{ham4mfd} R. S. Hamilton, \emph{Four-manifolds with positive curvature operator.} J. Differential Geom. {\bf 24} (1986), no. 2, 153--179.

\bibitem{hamharnack} R.S. Hamilton, \emph{The Harnack estimate for 
the Ricci flow.\/} J. Differential Geom. {\bf 37} (1993) 225--243.

\bibitem{KN} S. Kobayashi and K. Nomizu, \emph{ Foundations of differential geometry. Vol I.} Interscience Publishers, New York-London (1963) xi+329 pp.

\bibitem{formations} R. S. Hamilton, \emph{The formation of 
singularities
in the Ricci flow.\/} Surveys in differential geometry, Vol. II
(Cambridge, MA, 1993) 7--136, Internat. Press, Cambridge, MA, 1995.


\bibitem{MT} R.J. McCann and P.M. Topping, \emph{Ricci flow, entropy
and optimal transportation.\/} To appear, Amer. J. Math.
\url{http://www.warwick.ac.uk/~maseq}

\bibitem{micallef_moore} M.J. Micallef and J.D. Moore, 
\emph{Minimal two-spheres and the topology of manifolds 
with positive curvature on totally isotropic two-planes.\/} 
Ann. of Math. {\bf 127} (1988) 199--227.

\bibitem{P1} G. Perelman \emph{The entropy formula for the Ricci flow 
and its geometric applications.\/}
\url{http://arXiv.org/abs/math/0211159}v1  (2002).


\bibitem{RFnotes} P.M. Topping, \emph{Lectures on the Ricci flow.\/}
L.M.S. Lecture notes series {\bf 325} C.U.P. (2006)
\url{http://www.warwick.ac.uk/~maseq/RFnotes.html}

\bibitem{Lopt} P.M. Topping, \emph{$\cl$-optimal transportation
for Ricci flow.\/} 
J. Reine Angew. Math. {\bf 636} (2009) 93--122.

\bibitem{grenoble} P.M. Topping, \emph{Ricci Flow: 
The Foundations via Optimal Transportation} To be offered to
`S\'eminaires et Congr\`es.' S.M.F. (2009)
\url{http://www.warwick.ac.uk/~maseq}

\bibitem{wilking} B. Wilking, private communication (2009).

\end{thebibliography}
\end{document}